\PassOptionsToPackage{dvipsnames}{xcolor}
\PassOptionsToPackage{noabbrev}{cleveref}
\ifdefined\SIAMreview
  \documentclass[review,oneeqnum,onefignum,onetabnum]{siamonline250211}
\else
  \documentclass[oneeqnum,onefignum,onetabnum]{siamonline250211}
\fi

\usepackage{amssymb,mathtools}
\usepackage{booktabs,tabularx,array}
\usepackage{enumitem}
\usepackage{microtype}
\usepackage{tikz}
\usepackage[round,authoryear]{natbib}
\usepackage{xurl}
\ifdefined\SIAMreview
\fi
\hypersetup{colorlinks=true,linkcolor=MidnightBlue,citecolor=ForestGreen,
            urlcolor=MidnightBlue,
            pdftitle={Squaring Up by Selection: NP-Completeness at Three Simple Roots},
            pdfkeywords={polynomial systems, numerical algebraic geometry,
              homotopy continuation, matroid intersection, mixed volume,
              NP-completeness}}

\theoremstyle{plain}
\theoremheaderfont{\normalfont\sffamily}
\theorembodyfont{\normalfont}
\theoremseparator{.}
\theoremsymbol{}
\renewtheorem{definition}[theorem]{Definition}
\newsiamremark{example}{Example}
\newsiamremark{problem}{Problem}
\crefname{problem}{Problem}{Problems}
\Crefname{problem}{Problem}{Problems}
\newsiamremark{remark}{Remark}
\AddToHook{env/example/begin}{\crefalias{theorem}{example}}
\AddToHook{env/remark/begin}{\crefalias{theorem}{remark}}

\newcommand{\C}{\mathbb C}
\newcommand{\Q}{\mathbb Q}
\newcommand{\R}{\mathbb R}
\newcommand{\MV}{\operatorname{MV}}
\newcommand{\MVp}{\operatorname{MV}^{+}}
\newcommand{\SMV}{\operatorname{SMV}}
\newcommand{\supp}{\operatorname{supp}}
\newcommand{\conv}{\operatorname{conv}}
\newcommand{\rank}{\operatorname{rank}}
\newcommand{\iso}{\operatorname{Iso}}
\newcommand{\Bases}{\operatorname{Bases}}
\newcommand{\Vol}{\operatorname{Vol}}

\newcommand{\paperauthor}{Oren Bassik}
\newcommand{\paperaffiliation}{CUNY Graduate Center}
\newcommand{\paperemail}{obassik@gc.cuny.edu}
\newcommand{\paperorcid}{0009-0003-4354-127X}
\newcommand{\paperfunding}{}
\newcommand{\paperacknowledgments}{}

\title{Squaring Up by Selection: NP-Completeness at Three Simple Roots}
\ifx\paperfunding\empty\else
  \title{Squaring Up by Selection: NP-Completeness at Three Simple Roots%
    \thanks{\funding{\paperfunding}}}
\fi
\author{\paperauthor\thanks{\paperaffiliation%
  \ifx\paperemail\empty\else\ (\email{\paperemail})\fi.%
  \ifx\paperorcid\empty\else\ ORCID: \href{https://orcid.org/\paperorcid}{\paperorcid}.\fi}}
\headers{Squaring Up by Selection}{\paperauthor}
\hypersetup{pdfauthor={\paperauthor}}

\begin{document}
\maketitle

\begin{abstract}
To solve an overdetermined polynomial system numerically, one first
makes it square, usually by replacing the given equations with
as many random linear combinations as there are unknowns.  This is a
provably safe step, but it can substantially enlarge the supports.  The
alternative is to keep that many of the given equations themselves.
Selection preserves sparsity but risks geometry: a genuine solution can
cease to be an isolated point of the subsystem's zero set.  We show
that deciding whether a safe choice exists is NP-complete, already for
an explicit family of systems of degree three with radical ideal and
exactly three simple rational solutions.  For strong selection with the
nondegenerate rational solutions supplied explicitly, three is the exact
threshold when degrees are polynomially bounded: one or two solutions
reduce to matroid intersection, three already give NP-completeness.
Even without a degree bound, an arbitrarily long list never takes
the decision problem beyond NP.
On the hard family, five natural notions of a faithful subsystem
coincide, and every failing choice fails visibly: its zero set contains
an affine subspace through one of the three solutions.  A degree-four
variant shows that cost information does not help: every candidate that could
possibly succeed has mixed volume exactly three, and the problem is
NP-complete still.  The construction realizes Karp's three-dimensional
matching problem as the selection of a square subsystem from the given
equations.
\end{abstract}

\begin{keywords}
Overdetermined polynomial systems, numerical algebraic geometry,
homotopy continuation, NP-completeness, matroid intersection, mixed
volume.
\end{keywords}

\begin{MSCcodes}
Primary 14Q20, 68Q25; secondary 05B35, 52A39, 65H14.
\end{MSCcodes}

\section{Squaring up}\label{sec:intro}

An overdetermined polynomial system has more equations than unknowns
but may nevertheless have finitely many common zeros.  Standard
homotopy-continuation methods for computing isolated complex roots are
organized around square systems: $n$ equations in $n$ unknowns.  The
input instead supplies $m>n$, often because several exact constraints
describe the same underlying geometry.  Somehow $m$ equations must become $n$.

Here ``overdetermined'' is meant in the exact algebraic sense.  We
assume that the equations have common zeros and seek all of their
isolated complex solutions.  For noisy or inconsistent equations,
nonlinear least squares and Gauss--Newton methods provide a different
natural formulation \citep[Sec.~4]{HauensteinRegan2018}.  When the
minimum residual is zero, its global minimizers are the common zeros,
but least squares neither enumerates all such zeros nor by itself
supplies the square polynomial system used by the homotopy methods
considered here.  Our question is whether the given equations can be
reduced to a square subsystem while preserving their exact solution
geometry.

We prove that this question already becomes NP-complete for explicit
degree-three systems with radical ideal and exactly three simple
rational solutions.  For strong selection with the nondegenerate
solutions supplied as part of the input and polynomially bounded degrees,
this is the exact transition:
one or two solutions are handled by matroid intersection, while three
are enough for hardness.  Thus the obstruction appears before
singularity, high degree, or a large solution set can be blamed.

The standard approach to overdetermined systems is already described in the foundational papers of
numerical algebraic geometry
\citep{SommeseWampler1996,SommeseVerschelde2000}: randomization.
Replace $F=(f_1,\dots,f_m)$ by $n$ random linear
combinations $R\,F$ with $R\in\C^{n\times m}$.  For generic $R$, every
isolated solution of $F$ survives as an isolated solution of the
square system, nonsingular if it was nonsingular for $F$, and only
finitely many extraneous solutions appear
(\citealp[Sec.~13.5]{SommeseWampler2005};
\citealp[Thm.~7.2]{WamplerSommese2011}); these are discarded afterward by
evaluating the original equations.  The method is simple, provably safe, and built into the
standard software
\citep{BatesEtAl2013,BreidingTimme2018,Verschelde1999PHCpack}.

Randomization, however, will typically destroy sparsity.  A generic combined equation carries the
monomials of all $m$ originals, so their Newton polytopes are replaced
by a common hull.  The polytope-sensitive root count tracked by a
polyhedral solver (the mixed volume, discussed in \cref{sec:cost})
can consequently grow enormously.  Here is a particularly simple instance.

\begin{example}\label{ex:inflation}
The system
\[
 F=(x_1-1,\ \dots,\ x_n-1,\ x_1^d+\dots+x_n^d-n)
\]
with an integer $d\ge1$ has the single solution $(1,\dots,1)$.  Keep the first $n$ equations and
you have a square subsystem with the same solutions and even the same
ideal, since
$x_1^d+\dots+x_n^d-n=\sum_i(x_i^{d-1}+\dots+x_i+1)(x_i-1)$ is a
combination of the others; its mixed volume is $1$, and a polyhedral
solve tracks one
path.  Randomize instead, and each combined equation has Newton polytope
$\conv\{0,d\,e_1,\dots,d\,e_n\}$; the mixed volume is $d^{\,n}$, and the
solver tracks $d^{\,n}$ paths to find one point.
\end{example}

The sparsity-preserving alternative is to keep $n$ of the equations
\emph{themselves}.  Choose an index set $A$ of size $n$,
solve the square subsystem $F_A=(f_i)_{i\in A}$, and discard the extra
solutions, as before, by evaluating the $m-n$ equations that were
dropped. 
 Practitioners do exactly this when they can
\citep{Duff2021,HauensteinRegan2018}, and the certification literature
routinely works through square subsystems
\citep{DuffHeinSottile2022}.

A concrete example comes from parameter estimation in ODE models.  The
differential-algebra approach differentiates the model's output
equations repeatedly and substitutes derivative values estimated from
the measured data, leaving a polynomial system in the unknown
parameters and states; its solutions are the parameter and state
values consistent with the observations, and identifiability analysis
predicts how many there are \citep{LjungGlad1994,HongEtAl2020}.
Prolongation supplies more equations than unknowns, and the estimation
software selects a square subsystem greedily, by Jacobian rank, solves
it by homotopy continuation, and filters the extra solutions afterward
\citep{BassikEtAl2026}.  Randomization has a second cost here, beyond
support inflation.  The equations carry derivative estimates of
unequal accuracy, and a combination spreads the worst estimate into
every equation, while a selection can prefer the reliable low-order
ones.

The same selection appears in the five-point relative-pose problem of
computer vision, where a depth formulation has ten polynomial
equations in nine unknowns after dehomogenization, and the homotopy solver
keeps nine
\citep[Sec.~7.1]{HrubyEtAl2022}.  The subproblem recurs at every
iteration of a RANSAC loop, so support size and path count are paid at
every solve.  In both
applications, which equations to keep is decided on the exact
polynomial system, independently of the statistical question of
handling noise in the observations.

Because the selected equations are unchanged, no solution is lost as a
point: every zero of $F$ is a zero of $F_A$.  The question is which $A$
to take.  Two examples show the principal ways a choice can fail.

\begin{example}\label{ex:cost}
Suppose we choose greedily, by cost.  Consider
\[
 f_1=(x-1)(y-1),\qquad
 f_2=(x-1)(y-2),\qquad
 f_3=(y-3)(y-4)+(x-1)^2,
\]
three conics in two unknowns, with exactly two solutions, $(1,3)$ and
$(1,4)$, both nonsingular.  The three candidate pairs have mixed volumes
$2$, $4$, and $4$, so the greedy chooser takes the cheap pair
$\{f_1,f_2\}$, but that choice fails geometrically.
The two equations share the factor $x-1$: their common zero set is the
entire line $x=1$; and both solutions lie on that line.  Neither is an
isolated point of the subsystem's zero set, and a numerical solve of the
pair is under no obligation to find either one.

Note what did \emph{not} give the failure away.  By Bernstein's theorem
\citep{Bernshtein1975}, a square system can have at most as many
isolated solutions in $(\C^\times)^2$ as its mixed volume; the cheap
pair has mixed volume $2$, exactly the number of solutions it needs to
hold.  Its capacity is right; only its geometry is wrong.  The two pairs
that work cost $4$  and their extra solutions are
non-real points that the dropped equation filters out exactly as
intended.  Replacing the exponent $2$ by $d$ drives the ratio to
$(d+2)/2$: the cheap choice stays wrong, and the right choices get
arbitrarily more expensive.  Thus minimizing cost can select precisely
the subsystem that fails.
\end{example}

\begin{example}\label{ex:ws}
Perhaps, then, one should forget cost and simply search for a subsystem
that keeps every solution isolated.  Now comes the second surprise, an
example published by \citet[eq.~(70)]{WamplerSommese2011}.  With
$a=x-1$, $b=y-1$, $c=x+y-3$, take
\[
 F=(ab,\ ac,\ bc):
\]
three equations, two unknowns, and the three solutions
$(1,1)$, $(1,2)$, $(2,1)$, each nonsingular for the full system.  There
are three ways to choose a square subsystem, and all three fail.  The
pair $(ab,ac)=a\cdot(b,c)$ vanishes on the whole line $a=0$, which
passes through $(1,1)$ and $(1,2)$; the other two pairs fail the same
way with the roles of the lines exchanged (\cref{fig:ws}).  No selection
keeps the roots isolated, although every root is nondegenerate.
Randomization, by contrast, handles the example directly:
equation (71) of the same paper shows that two generic combinations
recover all three solutions at the price of one discardable extra.
Selection and randomization genuinely differ: a selection preserving
isolation need not exist even when every solution is nonsingular.
\end{example}

\begin{figure}[t]
\centering
\begin{tikzpicture}[scale=1.15,line width=1pt]
  \draw[->,gray] (-0.5,0) -- (3.5,0) node[below right] {$x$};
  \draw[->,gray] (0,-0.5) -- (0,3.5) node[above left] {$y$};
  \draw (1,-0.35) -- (1,3.3);
  \node[right] at (1.02,3.12) {$a=0$};
  \draw (-0.35,1) -- (3.3,1);
  \node[above] at (0.45,1.02) {$b=0$};
  \draw (-0.25,3.25) -- (3.25,-0.25);
  \node[right] at (2.42,0.72) {$c=0$};
  \fill (1,1) circle (2.2pt);
  \node[below left] at (1,1) {$(1,1)$};
  \fill (1,2) circle (2.2pt);
  \node[below left] at (1,2) {$(1,2)$};
  \fill (2,1) circle (2.2pt);
  \node[above right] at (2.05,1.05) {$(2,1)$};
\end{tikzpicture}
\caption{The example of \citet{WamplerSommese2011}.  Each equation of
$F=(ab,ac,bc)$ vanishes on two of the three lines, and the three
solutions (dots) are the pairwise intersections of the lines.  Any two
equations share a line, and each line carries two of the three
solutions: whichever pair one selects, two roots are handed to a
positive-dimensional component and cease to be isolated.}
\label{fig:ws}
\end{figure}
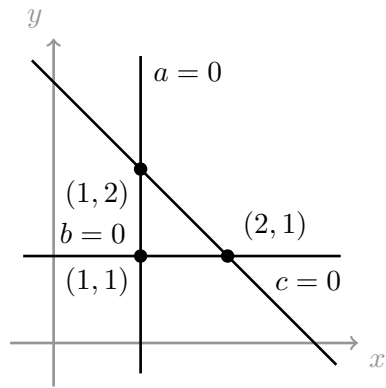

The examples lead to the question of this paper: given the equations,
can we decide whether some square subsystem keeps every solution
isolated?

Our main result, \cref{thm:main}, proves this decision problem
NP-complete on a recognizable family of systems of degree at most
three, with radical ideal and exactly \emph{three simple rational
solutions} known in advance.  In plain terms, a proposed selection is
quickly checked on this family, but a polynomial-time algorithm deciding
whether one exists would imply $\mathrm P=\mathrm{NP}$.  For strong
selection with explicitly listed nondegenerate solutions, three is the
exact threshold: one or two
solutions reduce to matroid intersection (\cref{sec:validity}), while
even without a degree bound an arbitrarily long list never takes the
problem beyond NP (\cref{sec:sharpness}).  The example of
\cref{ex:ws} is the smallest geometric representative of the transition.

The main theorem also closes both escape routes the examples above
might suggest.  On the hard family there is no gap between the many
things ``valid'' might mean: keeping solutions isolated, Jacobian
nonsingularity, finiteness of the subsystem's zero set, equality of
solution sets, equality of ideals --- all five coincide, for every
selection, and every failing selection fails for the visible reason of
\cref{ex:ws}, by containing an explicit affine subspace through one of
the three solutions (\cref{sec:main}).  The hardness is a fact about
geometry, not about any particular algebraic test.  And cost
information cannot rescue the greedy chooser: on a degree-four variant
of the family, every candidate that could possibly work has mixed
volume exactly $3$, and the decision problem is NP-complete still
(\cref{sec:cost}).

\begin{table}[t]
\small\centering
\begin{tabularx}{\textwidth}{@{}
  >{\raggedright\arraybackslash}X
  >{\raggedright\arraybackslash}X
  >{\raggedright\arraybackslash}p{0.16\textwidth}
  >{\raggedright\arraybackslash}p{0.19\textwidth}@{}}
\toprule
How the solutions are given & Task & Answer & Where\\
\midrule
one or two listed nondegenerate points & strong selection & polynomial time &
  \cref{rem:smallk}\\
three listed points & selection on the hard family & NP-complete &
  \cref{thm:main,thm:constantmv}\\
a nondegenerate list, any length & strong selection & in NP &
  \cref{prop:ceiling}\\
one point, one square system & decide isolation & open &
  \cref{prob:pointisolation}\\
\bottomrule
\end{tabularx}
\caption{The complexity map for explicitly listed solutions.  This
paper proves the first three rows; the first assumes polynomially bounded
degrees.  The final row is the local
isolation problem underlying weak validity and remains open; see
\cref{sec:open}.}
\label{tab:map}
\end{table}

\Cref{tab:map} shows where the results sit.  The combinatorial core
inside the main theorem is a problem that has been hard since the
beginning of complexity theory: three-dimensional matching, one of
Karp's original NP-complete problems \citep{Karp1972,GareyJohnson1979}.
In matroid language this is the common-basis problem for three
partition matroids (with two matroids being the polynomially solvable
matroid intersection problem \citep{Edmonds1979,Lawler1975}).  What the
present paper contributes is the transfer.  The construction realizes
this combinatorial obstruction by polynomial equations kept exactly as
supplied, in degree three, with full control of the ideal, the
solutions, and the failure geometry.  The hardness therefore lands on
precisely the decision the numerical analyst faces, and the natural
escape routes are all closed: singularity, degeneracy, cost signals, or a different
notion of correctness.

The paper is organized as the examples suggest.  \Cref{sec:validity}
says precisely what a selection must preserve; \cref{sec:construction} gives
the general construction behind examples like \cref{ex:ws};
\cref{sec:main} proves the main theorem; \cref{sec:sharpness} shows that
three is the exact threshold for strong selection with explicitly
listed nondegenerate solutions, and proves
the NP ceiling;
\cref{sec:cost} proves that cost cannot help; \cref{sec:open}
surveys the wider map and collects open problems.

\section{What a selection must preserve}\label{sec:validity}

Let us fix notation.  Throughout,
$F=(f_1,\dots,f_m)\in\Q[x_1,\dots,x_n]^m$ with $m>n$ is an
overdetermined system whose complex solution set
$S=V_{\C}(F)\subseteq\C^n$ is finite and nonempty.  A \emph{selection}
is an index set $A\subseteq\{1,\dots,m\}$ with $|A|=n$; it determines
the square subsystem $F_A=(f_i)_{i\in A}$ with zero set
$Z_A=V_{\C}(F_A)$.  Since the selected equations are unchanged,
$S\subseteq Z_A$ always: selection can add solutions, never remove one.
What it can do instead --- the examples of \cref{sec:intro} show --- is
worse: it can absorb a genuine solution into a positive-dimensional
component of $Z_A$, where a numerical solve has no duty to present it as
an isolated endpoint.

There are several inequivalent ways to demand that this not happen.
Five will matter.

\begin{definition}\label{def:predicates}
Write $\iso(Z_A)$ for the set of isolated points of $Z_A$.  The
selection $A$ is
\begin{enumerate}[label=(\roman*),itemsep=1pt]
\item \emph{weakly valid} if $S\subseteq\iso(Z_A)$ --- every genuine
      solution stays isolated;
\item \emph{strongly valid} if $\det DF_A(p)\neq0$ for every $p\in S$
      --- every genuine solution is a nonsingular zero of the
      subsystem;
\item \emph{globally finite} if $\dim Z_A=0$ --- the subsystem has
      finitely many zeros altogether;
\item \emph{set faithful} if $Z_A=S$ --- no extra zeros at all;
\item \emph{scheme faithful} if $(F_A)=(F)$ as ideals of
      $\Q[x_1,\dots,x_n]$ --- the subsystem carries the same algebra,
      not just the same zeros.
\end{enumerate}
\end{definition}

Only the following implications hold in general:
\[
 \text{strongly valid}\;\Longrightarrow\;\text{weakly valid},
\]
\[
 \text{scheme faithful}\;\Longrightarrow\;\text{set faithful}
 \;\Longrightarrow\;\text{globally finite}
 \;\Longrightarrow\;\text{weakly valid},
\]
the middle implication because $S$ is finite, and the last because
$S\subseteq Z_A$ and every point of a zero-dimensional
set is isolated.  None of the implications reverses without further
hypotheses, and one small system separates most of them at once: for
$F=(x^2-y,\ y,\ x)$, with the single nondegenerate solution at the
origin, the selection $\{1,2\}$ gives the ideal $(x^2-y,\,y)=(x^2,y)$
--- the origin is isolated with multiplicity two, so the selection is
weakly valid, globally finite, and set faithful, but neither strongly
valid (the two gradients are parallel there) nor scheme faithful
($(x^2,y)\neq(x,y)$).  The selection $\{2,3\}=(y,x)$ is valid in every
sense.

Weak validity is the correctness requirement of the solve-then-filter
workflow.  Strong validity is what one wants on top of correctness:
nonsingular endpoints are exactly the ones that Newton's method polishes
quadratically and that certification tools accept
\citep{HauensteinSottile2012,DuffHeinSottile2022}.  In general the two
notions differ; a pleasant feature of every hard family in this paper is
that on those families they will not.

The local meaning of weak validity connects the selection problem to
commutative algebra. 

\begin{proposition}\label{prop:bridge}
Let $p\in S$, let $\mathcal O_p$ be the local ring of $\C^n$ at $p$ and
$\mathfrak m_p$ its maximal ideal.  Then $p$ is isolated in $Z_A$ if and
only if $(F_A)\mathcal O_p$ is $\mathfrak m_p$-primary, that is, if and
only if the $n$ selected germs form a system of parameters of
$\mathcal O_p$.  Moreover $\det DF_A(p)\neq0$ if and only if the images
of the selected germs form a basis of
$\mathfrak m_p/\mathfrak m_p^2$.
\end{proposition}

\begin{proof}
The point $p$ is isolated in $Z_A$ iff
$\mathcal O_p/(F_A)\mathcal O_p$ has Krull dimension $0$, which for a
Noetherian local ring means $(F_A)\mathcal O_p\supseteq\mathfrak m_p^N$
for some $N$, that is, the ideal is $\mathfrak m_p$-primary.  And $n$
elements generating an
$\mathfrak m_p$-primary ideal of the regular $n$-dimensional local ring
$\mathcal O_p$ are, by definition, a system of parameters.  For the
second claim, $\det DF_A(p)\neq0$ says the differentials of the
selected germs span the $n$-dimensional cotangent space, and $n$
spanning vectors are a basis.
\end{proof}

So the selection problem asks for one supplied $n$-subset that is a
system of parameters simultaneously at every solution.  The word
\emph{simultaneously} is where the combinatorics enters.

Call $p\in S$ \emph{nondegenerate} if $\rank DF(p)=n$ --- this is what
``nonsingular for the full system'' meant in \cref{sec:intro}.  At a
nondegenerate solution the $m$ gradients
$\nabla f_1(p),\dots,\nabla f_m(p)$ span $\C^n$ and \emph{represent} a
matroid $M_p$ on the ground set $\{1,\dots,m\}$ of equation indices: a
subset is independent in $M_p$ when its gradients are linearly
independent, and the bases of $M_p$ --- its maximal independent sets
--- are exactly the $n$-subsets with linearly independent gradients
\citep[for matroids in general, see][]{Oxley2011}.

\begin{lemma}\label{lem:commonbasis}
Let $p\in S$ be nondegenerate and $|A|=n$.  Then $\det DF_A(p)\neq0$ if
and only if $A$ is a basis of $M_p$.  Consequently, if every point of
$S$ is nondegenerate,
\[
 A\ \text{strongly valid}
 \quad\Longleftrightarrow\quad
 A\in\bigcap_{p\in S}\Bases(M_p):
\]
strong selection is the common-basis problem for the $k=|S|$
matroids of the Jacobian.
\end{lemma}

\begin{proof}
By construction the bases of $M_p$ are the $n$-subsets whose gradients
are linearly independent, that is, those with $\det DF_A(p)\neq0$; the
second claim invokes \cref{def:predicates}(ii) at every solution at
once.
\end{proof}

This observation is definitional; we use
it as a dictionary, not as a result.  What the dictionary buys is the
well developed theory of common bases.  With one matroid, finding a basis is
easy.  With two, common-basis feasibility is the matroid intersection
problem, solved in polynomial time by \citet{Edmonds1979}; see also
\citet{Lawler1975}.  And the two-matroid case genuinely occurs: for
\[
 F=(y,\ \ x^2(x-1),\ \ x(x-1)^2),
 \qquad S=\{(0,0),\,(1,0)\},
\]
both solutions are nondegenerate and the two Jacobian matroids have
\emph{disjoint} basis families ($\{1,3\}$ at the first solution,
$\{1,2\}$ at the second), so no strongly valid selection exists ---
and matroid intersection detects this in polynomial time.

\begin{remark}\label{rem:smallk}
Suppose every solution is nondegenerate, the solutions are given as an
explicit list, and the equations have polynomially bounded degrees, so
that the gradients $\nabla f_i(p)$ can be written down with polynomially
many bits.  Then for $k=1$ strong-selection feasibility is a
matroid basis computation, and for $k=2$ it is matroid
intersection: polynomial time in both cases.  (The degree hypothesis
matters: for sparse input with binary-encoded exponents, even
evaluating a gradient exactly meets the obstruction of
\cref{rem:modular} below.)  For this listed strong-selection problem,
the transition begins at $k=3$.
\end{remark}

A final word on input, since complexity claims depend on how the input
arrives.  When we claim \emph{hardness}, we prove it for an explicit
family of systems recognizable in polynomial time, so the hardness
statement applies verbatim to arbitrary sparse rational input.  When we
claim \emph{membership} --- ``the problem is in NP'' --- we say exactly
which input model we mean: either that recognizable family, whose
syntactic form proves the standing hypotheses (finitely many solutions,
nondegeneracy, the count), or input that carries a list promised to
contain exactly the solutions.  For arbitrary sparse input we assert no upper bounds at
all; verifying the standing hypotheses is then itself a nontrivial
computational problem \citep[cf.][]{Koiran1996}, and whether
unrestricted weak selection even lies in NP is open
(\cref{sec:open}).

We now formalize the decision problems and their input models.

\begin{definition}\label{def:problems}
An instance of \textsc{Strong Selection} (respectively \textsc{Weak
Selection}) consists of an overdetermined system
$F\in\Q[x_1,\dots,x_n]^m$, $m>n$, encoded sparsely by rational
coefficients and binary integer exponents.  The promise is that
$S=V_{\C}(F)$ is finite and nonempty; the question is whether some
$n$-element selection $A$ makes $F_A$ strongly (respectively weakly)
valid with respect to $S$.

In \textsc{Listed Strong Selection} and \textsc{Listed Weak
Selection}, the input additionally contains a list of rational points
promised to be exactly $S$.  The listed strong variant further promises
that every point of $S$ is nondegenerate for the full system.  The
upper bounds and threshold statements for listed solutions refer to
these variants.  On each recognizable hard family constructed below,
the syntax itself establishes the promises, so restricting either
problem to that family gives an ordinary language rather than a
promise problem.
\end{definition}


\section{The slice construction}\label{sec:construction}

The failure in \cref{ex:ws} has a shape.  Three lines in the plane;
each equation vanishing on two of them; each pair of equations sharing
one; each shared line carrying two solutions.  This section gives a
construction that produces such geometry in any dimension, for any
number of solutions, from arbitrary prescribed linear algebra.  The
construction is the whole technical content of the paper.  The theorems
of \cref{sec:main,sec:sharpness,sec:cost} are three specializations of
it.

Here are the inputs.  Fix integers $q\ge1$, $k\ge1$, and
$\ell>q$; distinct rational numbers $\alpha_1,\dots,\alpha_k$ (the
\emph{slice values}); a base point $b\in\Q^q$; and matrices
$M^{(1)},\dots,M^{(k)}\in\Q^{q\times\ell}$, each of full row rank
$q$, with columns $a^{(j)}_i$.  In $n=q+1$ variables
$(x,s)=(x_1,\dots,x_q,s)$, the construction outputs the $m=\ell+1$ equations
\begin{equation}\label{eq:slicerows}
 r(s)=\prod_{j=1}^{k}(s-\alpha_j),
 \qquad
 g_i(x,s)=\sum_{j=1}^{k}L_j(s)\,
   \bigl\langle a^{(j)}_i,\;x-b\bigr\rangle
 \quad(1\le i\le\ell),
\end{equation}
where $L_j$ is the Lagrange basis polynomial taking the value $1$ at
$\alpha_j$ and vanishing at the other slice values.  Write
$F_M=(r,\ g_1,\dots,g_\ell)$.  Every equation has degree at most $k$.  The equation $r$, which we shall call the
\emph{guard}, pins $s$ to the slice values; on the slice $s=\alpha_j$
the row $g_i$ collapses to the linear form
$\langle a^{(j)}_i,x-b\rangle$, so the $j$-th slice sees exactly the
matrix $M^{(j)}$.  One polynomial system; $k$ prescribed
linear-algebra problems.

\begin{proposition}\label{prop:slices}
With the notation above:
\begin{enumerate}[label=(\alph*),itemsep=2pt]
\item $(F_M)=(r(s),\,x_1-b_1,\dots,x_q-b_q)$, a radical ideal, and
\[
 S=V_{\C}(F_M)=\{p_j=(b,\alpha_j):\ j=1,\dots,k\}
\]
consists of $k$ simple rational solutions.
\item Every selection omitting the guard contains the line
      $\{(b,s):s\in\C\}\subseteq Z_A$, which passes through every
      solution; such a selection is not weakly valid.
\item Let $A=\{r\}\cup A'$ with $A'\subseteq\{1,\dots,\ell\}$,
      $|A'|=q$ (abusing notation, we identify equations with their
      indices).  At every solution $p_j$,
\[
 \det DF_A(p_j)\;=\;\pm\,r'(\alpha_j)\cdot\det M^{(j)}_{A'},
\]
      where $M^{(j)}_{A'}$ is the square submatrix on the columns
      $A'$, and the following are equivalent:
      \begin{enumerate}[label=(\roman*)]
      \item $A$ is strongly valid;
      \item $A$ is weakly valid;
      \item $Z_A$ is finite;
      \item $Z_A=S$;
      \item $(F_A)=(F_M)$;
      \item $A'$ is a common basis of the $k$ matroids
            represented by the columns of
            $M^{(1)},\dots,M^{(k)}$.
      \end{enumerate}
\item If $\rank M^{(j)}_{A'}<q$ for some $j$, then
\[
 \Bigl(b+\ker\bigl((M^{(j)}_{A'})^{\!\top}\bigr)\Bigr)\times\{\alpha_j\}
 \ \subseteq\ Z_A
\]
      is an explicit rational affine subspace of dimension
      $q-\rank M^{(j)}_{A'}\ge1$ through the solution $p_j$; we call it
      the \emph{failure subspace} of the pair $(A,j)$.
\end{enumerate}
\end{proposition}

\begin{proof}
(a) Since the $\alpha_j$ are distinct, $r$ is squarefree, and the
Chinese Remainder Theorem splits the quotient by $(r)$ into $k$
independent copies of $\Q[x]$, one per slice:
\[
 \Phi:\ \Q[x,s]/(r)\ \xrightarrow{\ \sim\ }\
 \bigoplus_{j=1}^{k}\Q[x],
 \qquad
 h\longmapsto\bigl(h|_{s=\alpha_1},\dots,h|_{s=\alpha_k}\bigr).
\]
Under $\Phi$ the Lagrange polynomials become the idempotents of the
splitting, and
$\Phi(g_i)=\bigl(\langle a^{(1)}_i,x-b\rangle,\dots,
\langle a^{(k)}_i,x-b\rangle\bigr)$.  Ideals of a finite direct
sum decompose componentwise, and the $j$-th component of the ideal
generated by the $\Phi(g_i)$ is generated by the linear forms
$\langle a^{(j)}_i,x-b\rangle$, $1\le i\le\ell$ --- which, because
$M^{(j)}$ has row rank $q$, span all linear forms in $x-b$.  So each
component is $(x_1-b_1,\dots,x_q-b_q)$; pulling back, the ideal
$(F_M)$ contains $x_1-b_1,\dots,x_q-b_q$ together with $r$, and the reverse
containment is clear since each $g_i$ is a polynomial combination of
$x_1-b_1,\dots,x_q-b_q$.  Thus $(F_M)=(r,x-b)$.  The quotient
$\Q[x,s]/(r,x-b)\cong\Q[s]/(r)$ is a product of fields, so the ideal is
radical and its zero set is $\{(b,\alpha_j)\}$.  For simplicity of the
solutions: at $p_j$ we have $\nabla r(p_j)=(0,\dots,0,r'(\alpha_j))$
with $r'(\alpha_j)\neq0$, while
$\nabla g_i(p_j)=(a^{(j)}_i,\,0)$, because the $s$-derivative of $g_i$
is a combination of the terms
$\langle a^{(1)}_i,x-b\rangle,\dots,\langle a^{(k)}_i,x-b\rangle$, all of
which vanish at $x=b$.  The gradients $(a^{(j)}_i,0)$ span
$\Q^q\times\{0\}$ and the guard supplies the last direction: the full
Jacobian has rank $q+1=n$.

(b) Every $g_i$ vanishes identically on the line $x=b$, and the line
passes through each $p_j$.

(c) and (d).  At $p_j$ the selected Jacobian consists of the guard row
$(0,\dots,0,r'(\alpha_j))$ and the rows $(a^{(j)}_i,0)$ for $i\in A'$;
Laplace expansion along the guard row gives the displayed determinant
factorization, which is (i)$\Leftrightarrow$(vi).  If $A'$ is a common
basis, then in each component of the splitting the selected forms
already span all linear forms in $x-b$, so the argument of (a), applied
verbatim to the subsystem, gives $(F_A)=(r,x-b)=(F_M)$: that is (v),
and (v)$\Rightarrow$(iv)$\Rightarrow$(iii)$\Rightarrow$(ii) since the
common zero set is the finite set $S$ and every point of a finite set
is isolated.  If instead some $M^{(j)}_{A'}$ is singular, then on the
slice $s=\alpha_j$ the selected equations reduce to the guard (which
vanishes there) and the linear system
$(M^{(j)}_{A'})^{\!\top}(x-b)=0$, whose solution set through $b$ has
dimension $q-\rank M^{(j)}_{A'}\ge1$.  This is the affine subspace of
(d); it lies in $Z_A$, passes through $p_j$, and in one stroke destroys
isolation of $p_j$, finiteness, set faithfulness, and scheme
faithfulness, while the determinant factorization gives
$\det DF_A(p_j)=0$.  So each of (i)--(v) fails when (vi) fails, and the
equivalences are complete.
\end{proof}

The proposition cuts both ways: \emph{every} common-basis problem over
$\Q$ is realized.  And look again at \cref{fig:ws}: the shared
lines of the Wampler--Sommese example are exactly the objects that part
(d) produces, failure subspaces through genuine solutions.
The construction did not invent a new failure mode; it merely generalized the existing one.

\section{Three roots are enough}\label{sec:main}

Now we specialize.  An instance of \emph{three-dimensional matching}
is a list $T$ of triples $(t_1,t_2,t_3)\in\{1,\dots,q\}^3$, where
$\{1,\dots,q\}$ labels three separate ``parts''; a \emph{perfect
matching} is a set of $q$ triples that uses every label of every part
exactly once.  Deciding whether a perfect matching exists is one of the
original NP-complete problems \citep{Karp1972}; see
\citet{GareyJohnson1979}.  In matroid language, a set of $q$ triples is
a perfect matching exactly when it is a common basis of the three
\emph{partition} matroids of $T$ --- the matroids on ground set $T$ in
which a set of triples is independent when its first (respectively
second, third) coordinates are pairwise distinct.  Three-dimensional
matching is thus the common-basis problem one step beyond matroid
intersection, and it is exactly the problem the construction will realize
with $k=3$.

Call an instance $(q,T)$ \emph{standard} if $q\ge2$, the list $T$ is
duplicate-free, each of the $q$ labels appears in every coordinate of
some triple, and $|T|>q$.  (Every instance reduces to a standard one in
polynomial time, as the proof of \cref{thm:main} will show in three
sentences.)  For a standard instance, apply the construction to the three
partition matroids: slice values $\alpha=(1,2,3)$, base point
$b=\mathbf 1=(1,\dots,1)$, and slice matrices with the $\ell=|T|$ columns
$a^{(j)}_t=e_{t_j}$, so that
$\langle a^{(j)}_t,x-\mathbf 1\rangle=x_{t_j}-1$.  Explicitly, the
output system $F_T$ has the $|T|+1$ equations
\begin{equation}\label{eq:FT}
 r(s)=(s-1)(s-2)(s-3),
 \qquad
 g_t=L_1(s)\,(x_{t_1}-1)+L_2(s)\,(x_{t_2}-1)+L_3(s)\,(x_{t_3}-1)
\end{equation}
for $t\in T$, in the $n=q+1$ variables $(x_1,\dots,x_q,s)$, with
\[
 L_1(s)=\tfrac12(s-2)(s-3),\qquad
 L_2(s)=-(s-1)(s-3),\qquad
 L_3(s)=\tfrac12(s-1)(s-2).
\]
Membership of a system in this family is a syntactic check: read off
$T$, then verify the conditions on $(q,T)$.  Multiplying the $g_t$ by
$2$ makes every coefficient an integer of absolute value at most $8$
without changing any zero set, ideal, or Jacobian rank; the guard's
coefficients are integers of absolute value at most $11$.

Since the instance is standard, every label appears in every
coordinate of some triple, so the three slice matrices have full row
rank and \cref{prop:slices} applies.  Its part (c) now reads:
a selection is valid (in any of the five senses)
exactly when it contains the guard and its $q$ triples form a perfect
matching.  The main theorem is at this point a matter of assembling the
pieces.

\begin{theorem}\label{thm:main}
The family of systems $F_T$, over standard three-dimensional matching
instances $(q,T)$, is recognizable in polynomial time and has the
following properties.
\begin{enumerate}[label=(\roman*),itemsep=2pt]
\item Every equation has degree at most $3$; the ideal $(F_T)$ is
      radical; and the solution set consists of exactly three simple
      rational points, namely $(\mathbf 1,1)$, $(\mathbf 1,2)$,
      $(\mathbf 1,3)$.
\item For every selection $A$, the five properties of
      \cref{def:predicates} are equivalent.
\item Every invalid selection contains an explicit rational affine
      subspace of positive dimension inside $Z_A$ passing through one
      of the three solutions, computable from $A$ by linear algebra.
\item Deciding whether a valid selection exists is NP-complete; for
      arbitrary sparse rational input, \textup{\textsc{Strong
      Selection}} and \textup{\textsc{Weak Selection}} (and their
      listed variants) are NP-hard,
      already for systems with three simple rational solutions and
      degree at most $3$.
\end{enumerate}
\end{theorem}

\begin{proof}
Parts (i)--(iii) are \cref{prop:slices} applied to the standard
instance: the solutions are $(\mathbf 1,j)$ for $j=1,2,3$; the degrees
are at most $3$ ($L_j$ quadratic times a linear form); the failure
subspaces are the line $\{(\mathbf 1,s)\}$ when the guard is omitted
and $\bigl(\mathbf 1+\ker((M^{(j)}_{A'})^{\!\top})\bigr)\times\{j\}$
when a slice loses rank.

For (iv), membership in NP: a certificate is a selection consisting of
the guard and $q$ triples; the verifier checks that the three
coordinate multisets of the chosen triples are each exactly
$\{1,\dots,q\}$,
which by \cref{prop:slices}(c) is equivalent to validity in all five
senses, selections omitting the guard being invalid by
\cref{prop:slices}(b).  Hardness: we reduce from three-dimensional
matching, and an arbitrary instance may be assumed standard WLOG.  To see that, starting from an 
arbitrary instance, first merge
duplicate triples. An uncovered label rules out a perfect matching
outright.  When $|T|\le q$, the only conceivable matching is $T$
itself, checked by inspecting its three coordinate multisets.
Instances settled by these observations map to a fixed standard
instance with the same answer --- say
$\{(1,1,1),(2,2,2),(1,2,2)\}$ for YES and
$\{(1,1,1),(2,2,1),(2,1,2)\}$ for NO, both over $\{1,2\}$; the latter has
no matching because each of its three pairs repeats a label in some
coordinate.  Having standardized, now build $F_T$ (for the listed variants, include the list of 
its three solutions $(\mathbf 1,j)$) and observe that a valid selection exists
precisely when the instance has a perfect matching.  The system has
$O(|T|)$ equations, $O(1)$ monomials per equation, and coefficients of
absolute value at most $11$ so the whole reduction is polynomial.
Hardness for the recognizable family transfers to arbitrary sparse
input because the family's systems \emph{are} sparse rational systems.
\end{proof}

What, exactly, is happening at the three
solutions?  The three slices are three copies of the
same coordinate space; a triple row touches the three variables
$x_{t_1},x_{t_2},x_{t_3}$, one per slice; a selection of $q$ triple
rows is asked to be, at once, a basis at $s=1$, a basis at $s=2$, and a
basis at $s=3$.  That is all three-dimensional matching is.
\Cref{ex:ws} showed this mechanism in two variables: its three
shared lines are failure subspaces in the sense of
\cref{prop:slices}(d).

\section{Sharpness}\label{sec:sharpness}

Three questions of sharpness present themselves.  Is three roots
special, or an accident of the construction?  Does the difficulty keep
growing as the number of roots grows?  And is the NP upper bound as
obvious as it looks?  The answers are: special; no; and no.

For strong selection with explicitly listed nondegenerate solutions,
\cref{rem:smallk} handles one and two solutions in polynomial time,
and \cref{thm:main} proves hardness at three.  For larger fixed counts,
the construction extends the result without new ideas.

\begin{remark}\label{rem:higherk}
For every fixed $k\ge3$, the selection problem is NP-complete on a
recognizable family with \emph{exactly} $k$ simple rational
solutions and degree at most $k$.  Extend the standard instance by
$k-3$ further slices at $s=4,\dots,k$ whose slice matrices
have Vandermonde columns: the $i$-th triple row receives the column
$(1,i,\dots,i^{\,q-1})^{\!\top}$.  Any $q$ distinct Vandermonde columns form
a nonsingular matrix, so the new slices accept \emph{every} selection
as a basis and change nothing about which selections are valid, while
\cref{prop:slices} continues to apply.  The new entries
are bounded by $\ell^{\,q}$ and occupy $O(q\log\ell)$ bits, so the extended
system is still written down in polynomial size.  The difficulty is
unchanged; only the number of slices grows.
\end{remark}

The upper bound requires some care.  With the solutions given as a
list, one might try to verify a proposed selection by evaluating its
$k$ Jacobian determinants and checking that none is zero.  For the
families in this paper that works.  In general it does not, because
exact evaluation is not free.

\begin{remark}\label{rem:modular}
Consider $G=(x^N-y^N,\ x-2,\ y+2)$ with $N\ge2$ a large power of two.  The input
is written with about $\log_2N$ bits (the exponent is written in
binary), the single solution $(2,-2)$ is nondegenerate --- and the
determinant of one selected Jacobian at that solution is $N\,2^{N-1}$,
a number requiring $N$ bits.  The obvious verifier cannot even write the determinant down in
polynomial time.
\end{remark}

The cure is to verify modulo a small prime.

\begin{proposition}\label{prop:ceiling}
Suppose the input carries a list of rational points promised to be
exactly the solutions, each nondegenerate.  Then the existence of a strongly valid
selection --- \textup{\textsc{Listed Strong Selection}} --- is in NP, for
arbitrary overdetermined sparse rational systems.
\end{proposition}

\begin{proof}
A certificate is a selection $A$ together with a prime $\pi$ of
polynomially bounded bit length.  We first bound the integers whose
prime divisors must be avoided.  Let $L$ be the input length.  Every
coefficient, coordinate of a listed solution, and exponent has at most
$L$ bits, and there are at most $L$ variables and monomials.  Raising
an $L$-bit rational number to an exponent below $2^L$ produces a
rational number whose numerator and denominator have $2^{O(L)}$ bits.
Products over the variables, sums over the monomials, and determinants
of size at most $L$ preserve this bound.  Thus every evaluated
Jacobian determinant has numerator and denominator of at most
$2^{cL}$ bits for an absolute constant $c$.

The verifier first tests that $\pi$ is prime, which can be done in
polynomial time \citep{AgrawalKayalSaxena2004}.  It rejects if any
denominator appearing in an input coefficient or listed coordinate is
zero modulo $\pi$.  It then evaluates every entry
$\partial f_i/\partial x_j(p)$ in $\mathbb F_\pi$ by repeated squaring
and computes the determinants by Gaussian elimination.  Binary
exponents require $O(L)$ modular multiplications per power, and all
remaining dimensions are bounded by the input length, so the
verification is polynomial.  The verifier accepts exactly when
$\det DF_A(p)$ is nonzero modulo $\pi$ for every listed solution $p$.

If a determinant vanishes over $\Q$, it vanishes modulo every
admissible prime, so the verifier is sound.  Conversely, suppose all
$k$ determinants are nonzero.  Let $D$ be the product of the
nonzero numerators of these determinants and all nonzero denominators
of the input coefficients and listed coordinates.  The denominator of
an evaluated determinant introduces no new prime beyond the latter
denominators.  The height bound above gives $D$ at most $2^{c'L}$ bits,
so $D$ has fewer than $2^{c'L}$ distinct prime divisors.  Chebyshev's
elementary estimate gives more than $2^B/(2B)$ primes below $2^B$ for
all sufficiently large $B$.  Taking $B=(c'+2)L$ therefore leaves an
admissible prime of $O(L)$ bits that divides none of the relevant
integers.  For that prime every determinant remains nonzero, completing
the NP certificate.
\end{proof}

The results for listed solutions can now be summarized simply:  \emph{For solutions given as an explicit
list of nondegenerate rational points, strong selection is polynomial
at one or two solutions (when degrees are polynomially bounded),
NP-complete at every fixed number at least three, and never worse than NP.}
Weak selection inherits every hardness
statement, because on the hard families the two notions coincide. Its
membership in NP is a different matter, to which we return in
\cref{sec:open}: it hangs on the isolation problem,
\cref{prob:pointisolation}.  And the ceiling leaves only two places for
hardness beyond NP to come from: with the solutions listed, the problem
cannot exceed NP, so a harder variant must change how the solutions
enter --- implicitly, through a polynomial-size system with
exponentially many solutions, or through singular solutions, where
isolation itself is in question (\cref{sec:open}).

\section{What cost cannot see}\label{sec:cost}

Return to \cref{ex:cost}.  The cheap pair was wrong, and no count or
polytope betrayed it.  One might still hope that this was bad luck ---
that with more care, cost data could steer the choice after all.  This
section closes that  question. We construct a variant of the hard family so that every
candidate that could possibly be valid has \emph{the same} cost, known
in advance, but the selection problem is NP-complete anyway.

First, conventions.  For polytopes $P_1,\dots,P_n\subset\R^n$, the
volume $\Vol_n(\lambda_1P_1+\dots+\lambda_nP_n)$ is a homogeneous
polynomial in $\lambda_1,\dots,\lambda_n\ge0$, and the \emph{mixed
volume} $\MV(P_1,\dots,P_n)$ is its coefficient of
$\lambda_1\lambda_2\cdots\lambda_n$; with this normalization
$\MV(\Delta_n,\dots,\Delta_n)=1$ for the standard simplex.  Bernstein's
theorem \citep{Bernshtein1975}, together with the theorems of
Kushnirenko and Khovanskii \citep{Kushnirenko1976,Khovanskii1978}, says that a square
system whose supports have mixed volume $M$ has at most $M$ isolated
solutions in the torus $(\C^\times)^n$, with equality for generic
coefficients (this is the count a polyhedral homotopy tracks
\citep{HuberSturmfels1995}).  The count comes with two warnings.  The mixed volume of the
supports changes when the coordinates are translated, so all values
below refer to the coordinates as displayed.  And it counts only torus
solutions: the system $(x,\,y)$ has mixed volume $0$ and a perfectly
good solution at the origin.  For that reason two affine variants are
in use, the origin-augmented mixed volume $\MVp$, which bounds isolated
solutions in all of $\C^n$ \citep{LiWang1996,Rojas1994}, and the stable
mixed volume of \citet{HuberSturmfels1997}, which satisfies
$\MV\le\SMV\le\MVp$.  The theorem below pins all three at once, so no
choice among them matters.

Now the variant.  For a standard instance $(q,T)$, let
\[
 U(x)=\sum_{i=1}^{q}(x_i-1),
 \qquad
 \tilde g_t \;=\; g_t \;+\; 10\,r(s)\,U(x)\quad(t\in T),
 \qquad
 \tilde F_T=(r,\ \tilde g_t: t\in T).
\]
The degree rises to four.  Nothing else changes:

\begin{lemma}\label{lem:congruence}
For every standard $(q,T)$:
\begin{enumerate}[label=(\alph*),itemsep=1pt]
\item $\tilde g_t\equiv g_t \pmod{(r)}$; hence
      $(\tilde F_T)=(F_T)$, the solution set is again the three points
      $(\mathbf 1,j)$ with radical ideal, and for every selection
      $A\ni r$ the selected ideals agree.
\item $\nabla\tilde g_t=\nabla g_t$ at every solution, since
      $\nabla(rU)=U\nabla r+r\nabla U$ and both $r$ and $U$ vanish
      there.
\item Selections omitting the guard still contain the line
      $\{(\mathbf 1,s)\}$, since $g_t$ and $U$ both vanish on it.
\item Consequently the entire classification of
      \cref{prop:slices}\textup{(b)--(d)} holds verbatim for
      $\tilde F_T$: valid means guard plus perfect matching, with the
      same failure subspaces.
\end{enumerate}
\end{lemma}

\begin{proof}
(a) $\tilde g_t-g_t=10rU\in(r)$, and generating sets that differ by
multiples of a common generator generate the same ideal.  (b) and (c)
are the displayed observations; (d) follows since ideals, zero sets,
and Jacobians at the solutions all agree with those of $F_T$.
\end{proof}

What the modification buys is uniformity of the Newton polytopes.

\begin{lemma}\label{lem:support}
For every $t\in T$, the support of $\tilde g_t$ is
\[
 \mathcal A_q=\{k\,e_s:0\le k\le3\}\;\cup\;
 \{e_i+k\,e_s:\ 1\le i\le q,\ 0\le k\le3\},
\]
independently of $t$, while the guard has support
$\mathcal B=\{k\,e_s:0\le k\le3\}$.  Hence
$\conv(\mathcal A_q)=\Delta_q+[0,3e_s]$, a prism over the standard
simplex, $\conv(\mathcal B)=[0,3e_s]$, and both supports contain the
origin.
\end{lemma}

\begin{proof}
\Cref{app:support}.  The point is that the coefficient $10$ makes the
contribution of $10\,rU$ strictly dominate, monomial by monomial, every
coefficient that $g_t$ can produce --- including when a triple repeats a
label --- so no cancellation occurs anywhere, and every row's support
fills out to the same set $\mathcal A_q$.
\end{proof}

\begin{proposition}\label{prop:mvvalues}
For every $q\ge1$,
\[
 \MV\bigl(\conv\mathcal B,\ \conv\mathcal A_q,\dots,\conv\mathcal
 A_q\bigr)=3
 \qquad\text{and}\qquad
 \MV\bigl(\conv\mathcal A_q,\dots,\conv\mathcal A_q\bigr)=3(q+1),
\]
with one guard and $q$ modified rows in the first case and $q+1$
modified rows in the second.  In both cases the ordinary,
origin-augmented, and stable mixed volumes coincide.
\end{proposition}

\begin{proof}
Write $P=\Delta_q+[0,3e_s]$ and $Q=[0,3e_s]$.  Minkowski sums of these
bodies are again prisms:
\[
 \lambda_0Q+\sum_{i=1}^{q}\lambda_iP
 \;=\;
 \Bigl(\sum_i \lambda_i\Bigr)\Delta_q
 +\Bigl[0,\ 3\Bigl(\lambda_0+\sum_i \lambda_i\Bigr)e_s\Bigr],
\]
with volume $\bigl(\sum_i \lambda_i\bigr)^{q}/q!\ \cdot\
3\bigl(\lambda_0+\sum_i \lambda_i\bigr)$.  The coefficient of
$\lambda_0\lambda_1\cdots\lambda_q$ is $3\cdot q!/q!=3$.  With $q+1$
copies of $P$ the volume is $3\bigl(\sum\lambda\bigr)^{q+1}/q!$, whose
$\lambda_1\cdots\lambda_{q+1}$-coefficient is $3(q+1)!/q!=3(q+1)$.  Since both supports contain the origin,
augmenting by the origin changes no polytope, so $\MVp=\MV$; and
$\MV\le\SMV\le\MVp$ squeezes the stable value
\citep{HuberSturmfels1997}.
\end{proof}

\begin{theorem}\label{thm:constantmv}
On the family $\tilde F_T$, over standard instances (recognizable in
polynomial time, degree at most $4$, radical ideal, three simple
rational solutions, all in the torus):
\begin{enumerate}[label=(\roman*),itemsep=2pt]
\item the five properties of \cref{def:predicates} coincide for every
      selection, with the same explicit failure subspaces as before;
\item every selection containing the guard has ordinary $=$ augmented
      $=$ stable mixed volume $3$; every selection omitting the guard
      is invalid and has all three equal to $3(q+1)$;
\item every valid selection $A$ satisfies $Z_A=S$: the subsystem
      attains its Bernstein bound with the three genuine solutions and
      produces no extra solutions to filter;
\item deciding whether a valid selection of mixed volume at most $3$
      exists (in any of the three variants) is NP-complete; for
      arbitrary sparse input it is NP-hard.
\end{enumerate}
\end{theorem}

\begin{proof}
(i) is \cref{lem:congruence}(d).  (ii) is
\cref{lem:support,prop:mvvalues}: a selection containing the guard has
support tuple $(\mathcal B,\mathcal A_q,\dots,\mathcal A_q)$ and one
omitting it has $(\mathcal A_q,\dots,\mathcal A_q)$ --- by uniformity,
\emph{which} triples were selected is invisible to the polytopes.
(iii): a valid selection is scheme faithful by (i), so $Z_A=S$
consists of the three simple torus solutions; by (ii) the mixed volume
is $3$, so the Bernstein bound is attained, and set faithfulness rules
out solutions anywhere else in $\C^n$.  (iv) Membership: certify with a
guard-plus-matching selection as in \cref{thm:main}; the cost condition
holds automatically by (ii) and requires no computation.  Hardness: the
reduction of \cref{thm:main} composed with the modification --- a
perfect matching yields a valid selection of cost exactly $3$; no
matching, no valid selection at any cost.  Encoding sizes stay
polynomial (\cref{app:support}).
\end{proof}

The theorem separates the difficulty of deciding whether a valid
subsystem exists from the difficulty of \emph{costing} a candidate.
Computing mixed volumes exactly is itself hard in general (\#P-hard,
by \citet{DyerGritzmannHufnagel1998}), and one might have suspected
any hardness of cost-aware selection to be an artifact of that.  Here
the opposite happens: on this family the cost of every candidate is
known before the search begins, the objective is constant across the
entire class of potentially valid selections, and an exact
mixed-volume oracle would help not at all.  All the difficulty is in
the geometry of keeping roots isolated.  For the solve-then-filter
workflow, a valid selection tracks three paths, ends at the three true
solutions, certifies them by nonsingularity, and discards nothing.

One caution.  The theorem does \emph{not} say that minimizing mixed
volume over valid selections is hard because the objective is hard.
It says the minimization is NP-hard already when the objective is
trivial.  On families where the cost genuinely varies, the
optimization problem is unclassified even in the two-solution regime,
where feasibility is polynomial ---
see \cref{prob:minmv}.

\section{Open problems}\label{sec:open}

\Cref{tab:map} summarizes the classification proved here.  For strong
selection with explicitly listed nondegenerate rational solutions, the
problem is polynomial with one or two solutions when degrees are
polynomially bounded, NP-complete for every fixed number from three
onward, and in NP for a list of arbitrary
length.  Weak selection inherits the lower bounds because the notions
coincide on the hard families, but its general upper bound remains tied
to the local-isolation problem below.

The ceiling leaves two routes to hardness beyond NP
(\cref{sec:sharpness}).  The implicit route, where the solution set is
described by equations rather than listed, is not pursued here.  The
singular route is \cref{prob:singular} below.

A word of scale, finally: these are worst-case theorems.
Conditioning-based and adaptive selection heuristics
\citep{HauensteinRegan2018,Duff2021} may perform excellently on the
structured systems encountered in practice. What the theorems rule out,
unless P equals NP, is a method that is fast and correct on
\emph{every} instance.  They also say nothing about rewriting the
system: liftings that introduce auxiliary variables can lower the
mixed volume \citep{BorgerEtAl2025}, whereas selection takes the
equations as given.

One question sits underneath the classification rather than beyond it,
and we state it first.  The three problems after it are the more
plausible next steps.

\begin{problem}\label{prob:pointisolation}
Given a \emph{square} system $G\in\Q[x_1,\dots,x_n]^n$ and a rational
point $z$ with $G(z)=0$, what is the complexity of deciding whether $z$
is isolated in $V_{\C}(G)$?  This is the local question underlying weak
validity: every original solution must remain isolated in the selected
subsystem.  Related work includes \citet{Koiran2000LocalDimensions} on
local dimensions of constructible sets and the numerical local-dimension
test of \citet{BatesEtAl2009LocalDimension}.
\end{problem}

\begin{problem}\label{prob:singular}
Can hardness beyond NP be achieved with a \emph{constant} number of
solutions?  With the list of nondegenerate rational solutions supplied,
\cref{prop:ceiling} confines any such hardness to weak selection,
through singular local geometry of the selected subsystems, where weak
and strong validity part company and Jacobian rank no longer decides
isolation.
\end{problem}

\begin{problem}\label{prob:minmv}
On families where the cost varies and valid selections are known or can
be found efficiently, what is the complexity of minimizing mixed volume
over valid selections?  What polynomial-time approximation guarantees
are possible?  A valid selection with smaller mixed volume can reduce the
number of paths tracked by a polyhedral homotopy solver.  Related work
includes weighted matroid intersection for additive costs \citep{Frank1981}
and approximation of the mixed volume of a given tuple of convex bodies
\citep{Gurvits2009}.
\end{problem}

\begin{problem}\label{prob:structured}
For \textsc{Listed Strong Selection} with polynomially bounded degrees,
find structural classes of systems for
which valid selection is polynomial beyond $k\le2$ and a fixed
number $m-n$ of extra equations.  Examples could include support conditions, coefficient
conditions, or shared decompositions of the Jacobian matroids.  Parametric families are the natural place to look and a
selection certified once, offline, may be reusable across the entire
family; the prolonged systems of parameter estimation are exactly of
this kind, their structure fixed by the model while the data supply
only coefficients \citep{BassikEtAl2026}.
\end{problem}

Squaring up will remain a daily act in numerical algebraic geometry,
and randomization will remain safe.  What the theorems here add is a
boundary.  If one insists on keeping the equations as given, for example, in an effort to preserve sparsity, then the choice
of which to keep is a hard combinatorial problem already for systems
with three simple rational roots and radical ideal.  Neither cost
information nor a different choice among the natural correctness
criteria removes the obstruction on the constructed families.
\Cref{prob:minmv,prob:structured} ask which additional structures
restore tractability.

\appendix

\section{Supports and volumes}\label{app:support}

We prove \cref{lem:support}.  Work with the integer scaling
$2\tilde g_t=2g_t+20\,r(s)\,U(x)$; scaling a row changes no support.
The ingredients, fully expanded:
\[
 2L_1=s^2-5s+6,\qquad
 2L_2=-2s^2+8s-6,\qquad
 2L_3=s^2-3s+2,\qquad
 r=s^3-6s^2+11s-6,
\]
and, as a sanity check, $2L_1+2L_2+2L_3=2$, the partition of unity.
Write $\rho_3,\rho_2,\rho_1,\rho_0=1,-6,11,-6$ for the coefficients of
$r$.

Every monomial of $2\tilde g_t$ has $x$-degree at most $1$ and
$s$-degree at most $3$, so its support is contained in $\mathcal A_q$.
We show every monomial of $\mathcal A_q$ actually occurs.

\emph{Monomials $x_i s^k$.}  The term $20\,r\,U$ contributes
$20\rho_k$ to $x_i s^k$ for \emph{every} $1\le i\le q$ and $0\le k\le3$:
the values are $(20,-120,220,-120)$ for $k=(3,2,1,0)$.  The term
$2g_t$ contributes only when $i$ is a label of the triple $t$, and its
contribution is a sum of at most three $s^k$-coefficients drawn from
the $2L_j$ above, hence bounded in absolute value by $0$, $4$, $16$,
$14$ for $k=3,2,1,0$.  Since $20>0$, $120>4$, $220>16$, and $120>14$,
no cancellation is possible: every coefficient of $x_i s^k$ is
nonzero, and lies in $[-134,236]$.

\emph{Monomials $s^k$.}  The pure-$s$ part of $2g_t$ is
$-(2L_1+2L_2+2L_3)=-2$, a constant, while the pure-$s$ part of
$20\,r\,U$ is $-20q\,r(s)$, with coefficients
$(-20q,120q,-220q,120q)$; adding $-2$ at $k=0$ gives $120q-2\neq0$.
All four powers occur.

So $\supp(2\tilde g_t)=\mathcal A_q$ for every triple, repeated labels
included, and $\supp(r)=\mathcal B$ is immediate.  The convex hulls are
the prism and the segment: the prism's vertices
$\{(v,0),(v,3e_s):v\in\{0,e_1,\dots,e_q\}\}$ all lie in
$\mathcal A_q$, and $\mathcal A_q$ lies in the prism.

Finally, the normalization of the mixed volume can be checked against
the simplest case.  Our convention takes $\MV$ to be the coefficient
of $\lambda_1\cdots\lambda_n$ in
$\Vol_n(\lambda_1P_1+\dots+\lambda_nP_n)$; for
$P_i=\Delta_n$ this gives
$\Vol_n\bigl((\sum_i\lambda_i)\Delta_n\bigr)=(\sum_i\lambda_i)^n/n!$,
with $\lambda_1\cdots\lambda_n$-coefficient $n!/n!=1$, matching the one torus solution
of $n$ generic linear equations.  The $q=1$ case of
\cref{prop:mvvalues} can be checked by hand: the prism is then a
$1\times3$ rectangle, the guard a vertical segment of length $3$, and
$\Vol(\lambda_0Q+\lambda_1P)=\lambda_1\cdot3(\lambda_0+\lambda_1)$,
with $\lambda_0\lambda_1$-coefficient $3$.

\emph{Encoding.}  After scaling by $2$, each triple row of $F_T$ is
\[
 2g_t=(s^2-5s+6)(x_{t_1}-1)-(2s^2-8s+6)(x_{t_2}-1)+(s^2-3s+2)(x_{t_3}-1),
\]
with integer coefficients of absolute value at most $8$ after
collection and at most $10$ monomials; the guard is
$r=s^3-6s^2+11s-6$.  Using sparse lists of nonzero variable--exponent pairs
within monomials, the whole system takes $O(|T|)$ numbers of $O(1)$
bits plus indices of $O(\log q)$ bits.  For the modified system, each
row $2\tilde g_t=2g_t+20\,r\,U$ has exactly $4(q+1)$ monomials with
integer coefficients of absolute value at most $220q+2$ --- the four
pure powers of $s$ grow linearly in $q$, everything else stays in
$[-134,236]$, as computed above --- for a total encoding of
$O(|T|\,q)$ numbers of $O(\log q)$ bits.  Thus both families have
polynomial encoding length.

\ifx\paperacknowledgments\empty\else
\section*{Acknowledgments}
\paperacknowledgments
\fi

\section*{Declarations}

\emph{Code availability.}
The SymPy script \texttt{verify-three-root-family.py}, which checks
the explicit constructions and examples of this paper on small
instances, accompanies the manuscript together with its
reproducibility guide \texttt{VERIFICATION\_README.md}.

\emph{Use of artificial intelligence.}
OpenAI Codex and Claude Fable were used for mathematical and editorial
review, literature cross-checking, LaTeX preparation and diagnostics,
and code-execution assistance.  The author assumes responsibility for
all content.

\bibliographystyle{plainnat}
\bibliography{references}

\end{document}